\documentclass{article}

\usepackage{amsfonts}
\usepackage{amsmath}
\usepackage{amssymb}
\usepackage{amsthm}
\usepackage{mathrsfs}
\usepackage{enumerate} 
\usepackage{tikz} 
\usepackage{tikz-cd}
\usepackage[numbers]{natbib}
\usepackage{hyperref}
\usetikzlibrary{arrows}
\usepackage{enumitem} 
\setlist[enumerate]{label={\upshape(\roman*)}}

\newcommand{\GL}{\text{GL}} 
\newcommand{\Gal}{\mathrm{Gal}} 

\newcommand{\F}{\mathbb{F}}

\newcommand{\Q}{\mathbb{Q}}
\newcommand{\Z}{\mathbb{Z}}

\newcommand{\frakd}{\mathfrak{d}}
\newcommand{\frakD}{\mathfrak{D}}
\newcommand{\frakp}{\mathfrak{p}}

\newcommand{\calO}{\mathcal{O}}

\title{Discriminant-Stability in $p$-adic Lie Towers of Number Fields}
\author{James Upton}

\theoremstyle{plain}
\newtheorem{theorem}{Theorem}[section]
\newtheorem{lemma}[theorem]{Lemma}
\newtheorem{proposition}[theorem]{Proposition}

\newtheorem{conjecture}[theorem]{Conjecture}

\theoremstyle{definition}
\newtheorem{definition}[theorem]{Definition}
\newtheorem{example}[theorem]{Example}
\newtheorem{remark}[theorem]{Remark}

\makeatletter
\def\blfootnote{\gdef\@thefnmark{}\@footnotetext}
\makeatother

\begin{document}
	\maketitle
	\begin{abstract}
		In this paper we consider a tower of number fields $\cdots \supseteq K(1) \supseteq K(0) \supseteq K$ arising naturally from a continuous $p$-adic representation of $\Gal(\bar{\Q}/K)$, referred to as a $p$-adic Lie tower over $K$. A recent conjecture of Daqing Wan hypothesizes, for certain $p$-adic Lie towers of curves over $\F_p$, a stable (polynomial) growth formula for the genus. Here we prove the analogous result in characteristic zero, namely: the $p$-adic valuation of the discriminant of the extension $K(i)/K$ is given by a polynomial in $i,p^i$ for $i$ sufficiently large. This generalizes a previously known result on discriminant-growth in $\Z_p$-towers of local fields of characteristic zero.
	\end{abstract}
	\section{Introduction}
	
	Consider a smooth projective curve $C$ over a finite field of characteristic $p$. Let $S$ be a finite set of closed points of $C$ and $U=C\backslash S$. Let $\rho:\pi_1^\text{arith}(U)\to \GL_r(\Z_p)$ be a $p$-adic Galois representation, so that the image of $\rho$ is a compact $p$-adic Lie group. The congruence subgroups modulo high powers of $p$ constitute a filtration of this group by open normal subgroups, referred to as a \emph{Lie filtration} (see Section \ref{s:Lie} for a more precise description). Via the Galois correspondence, we obtain a \emph{$p$-adic Lie tower} of smooth projective curves $\cdots \rightarrow C_1 \rightarrow C_0 \rightarrow C$, which ramifies in the finite set $S$. We are inspired by the recent conjectures of Daqing Wan \cite{Wan} describing the stable arithmetic properties of $p$-adic Lie towers coming from geometry. In particular, we are interested in the genus-stability conjecture:
	
	\begin{conjecture}\label{conj}
		Let $g_i$ denote the genus of $C_i$, and suppose the representation $\rho$ is algebraic-geometric and ordinary (as in \cite{Wan}). Then the sequence $(g_i)$ is given by a polynomial in $p^i$ for $i\gg 0$.
	\end{conjecture}
	
	In this paper we will formulate and prove a characteristic-0 analogue of Conjecture~\ref{conj}. Namely, we will consider a $p$-adic Lie tower of number fields $\cdots \supseteq K(1) \supseteq K(0) \supseteq K$ arising from a $p$-adic Galois representation of $\Gal(\bar{\Q}/K)$. In this setting there is no direct analogue of ``genus,'' but we may instead consider the $p$-adic growth of the discriminant (or equivalently, the different) of the extensions $K(i)/K$ as $i\to \infty$. Recall that in the function field case, these two quantities are related by the Riemann-Hurwitz formula (\cite{Oort}, Theorem 6.2):
	\begin{align*}
		2g_i-2	&=	[C_i:C_0](2g_0-2)-\sum_{P \in C_i}v_P(\frakD_{C_i/C_0}).
	\end{align*}
	
	\begin{definition}
		Let $\cdots	\supseteq	K(1)	\supseteq	K(0)	\supseteq K$ be a $p$-adic Lie tower of local or global fields of characteristic $0$. We say the tower is \emph{$p$-stable} (or when $p$ is clear, just \emph{stable}) if for every prime $\frakp$ in $K$ with residue characteristic $p$,
		\begin{align*}
			v_{\frakp}(\frakd_{K(i)/K})	&=	f(i,p^i)
		\end{align*}
		for $i\gg0$. Here $\frakd_{K(i)/K}$ denotes the discriminant of $K(i)/K$, and $f \in \Q[x,y]$ is some polynomial.
	\end{definition}
	
	Our main result is
	
	\begin{theorem}\label{main}
		Every $p$-adic Lie tower of local or global fields of characteristic $0$ is $p$-stable.
	\end{theorem}
	
	In what follows we will see that Theorem~\ref{main} follows from the special case of a totally ramified tower of local fields. Thus in Section~\ref{s:local} we will reduce the problem to the local case, at which point we can state an explicit stability formula different. The proof proceeds in Section~\ref{s:proof}, where we reduce to the case of a totally ramified tower, and carry out the computation in terms of local ramification groups.
	
	\begin{remark}
		Tate previously considered the case of a totally ramified $\Z_p$-tower of local fields of characteristic $0$ in \cite{Tate}, deducing a polynomial formula for the discriminant up to $\calO(1)$. The stability of these towers was proven by Kosters in \cite{Kosters2}. As a consequence of the work of Kosters and Wan in \cite{Kosters}, the analogous characteristic-$p$ statement of Theorem~\ref{main} is known to be false in general. Thus the additional assumptions in Conjecture~\ref{conj} that the tower be algebraic-geometric and ordinary are not superfluous. The special case of Conjecture~\ref{conj} for $\Z_p$-towers of curves is proven in the forthcoming work of Joe Kramer-Miller.
	\end{remark}

	\section{Background on $p$-adic Lie Groups}\label{s:Lie}
	
		In this section we describe the construction of $p$-adic Lie towers over any field $K$, and the relevant background material on $p$-adic Lie groups. For a detailed reference, see~\cite{Dixon}. Let
		\begin{align*}
			\rho:	\Gal(K^s/K)	\to	\GL_n(\Z_p)
		\end{align*}
		be a continuous $p$-adic Galois representation, and let $L$ be the fixed field of $\ker(\rho)$. Then $G=\Gal(L/K)$ is a compact subgroup of $\GL_n(\Z_p)$, and this will give $G$ the structure of a $p$-adic Lie group (Theorem~\ref{t:subgroup} below). We will use the following algebraic characterization of such groups:
	
	\begin{definition}
		A pro-$p$ group $H$ is \emph{uniform} if
		\begin{enumerate}
			\item	$H$ is finitely generated
			\item	$H^p$ is a normal subgroup of $H$, and $H/H^p$ is abelian\footnote{When $p=2$, we require that $H/H^4$ is abelian. The proofs below are easily modified in this case.}
			\item	The $p^\text{th}$-power map induces a \emph{shift isomorphism}
				\begin{align*}
					H^{p^i}/H^{p^{i+1}}	\to	H^{p^{i+1}}/H^{p^{i+2}}.
				\end{align*}
		\end{enumerate}
		A topological group $G$ is a \emph{$p$-adic Lie group} if it admits an open, uniform subgroup $H$.
	\end{definition}
	
	Observe that for any uniform group $H$, the quotient $H/H^p$ is a finitely generated abelian group annihilated by $p$. By (3), it follows that
	\begin{align*}
		H^{p^i}/H^{p^{i+1}}	\cong	(\Z/p\Z)^d
	\end{align*}
	for all $i$ and some fixed $d$, which we refer to as the \emph{dimension} of $H$. If $G$ is a $p$-adic Lie group, then all open uniform subgroups of $G$ have the same dimension, which we refer to as the dimension of $G$.
	
	\begin{example}\label{e:gln}
		Suppose that $G=\GL_n(\Z_p)$, and let $\pi:G\to \GL_n(\Z/p\Z)$ be the reduction modulo $p$. Letting $H=\ker(\pi)$, we have that
		\begin{align*}
			H^{p^i}	&=	\ker(G\to	\GL_n(\Z/p^i\Z)).
		\end{align*}
		It is straightforward to verify that $H$ is an open, uniform subgroup of $G$, and that $\dim(G)=\dim(H)=n^2$.
	\end{example}
	
	We cite the following properties of $p$-adic Lie groups without proof:
	
	\begin{proposition}\label{props}~
		\begin{enumerate}
			\item	Let $H$ be a uniform pro-$p$ group of dimension $d$. Then for $i$ sufficiently large, and any $r\leq i+1$,
				\begin{align*}
					H^{p^i}/H^{p^{i+r}}	\cong (\Z/p^r\Z)^d,
				\end{align*}
				and the $p^\text{th}$ power map on $H$ induces an isomorphism
				\begin{align*}
					H^{p^i}/H^{p^{i+r}}	\to	H^{p^{i+1}}/H^{p^{i+1+r}}
				\end{align*}
			\item	Let $G$ be a compact $p$-adic Lie group. Then $G$ admits an open, normal uniform subgroup $H$, which necessarily has finite index in $G$.
		\end{enumerate}
	\end{proposition}
	\begin{remark}
		Statement (i) above is a consequence of the Lie group-Lie algebra correspondence, which is not discussed here. The correspondence is discussed in \cite{Probst} Section 3, and the statement is given later in Section 6.5. Statement (ii) is given in \cite{Dixon}, Theorem 8.34.
	\end{remark}
	
	Since the Galois group $G$ is always compact, it admits an open normal uniform subgroup $H$. Define a family of subgroups $G(i)	=	H^{p^i}.$ These are easily seen to be open and normal in $G$. Thus we obtain a filtration
	\begin{align*}
			G	&\geq	G(0)	\geq	G(1)	\geq	\cdots
	\end{align*}
	of $G$, which we refer to as a \emph{Lie filtration}. Via the Galois correspondence, there exists a decomposition of $L/K$ into a $p$-adic Lie tower
	\begin{align*}
		L	\supseteq	\cdots	\supseteq	K(1)	\supseteq	K(0)	\supseteq K.
	\end{align*}
	There is no canonical decomposition of the extension $L/K$, as these are in one-to-one correspondence with choices of $H$. In any case, we have that
	\begin{align*}
		[K(i):K]	&=	[K(i):K(0)][K(0):K]	=	p^{\dim(G)i}[K(0):K].
	\end{align*}
	
	Finally, we will prove the $p$-adic analogue of the classical ``closed-subgroup theorem'' for Lie groups. This ensures that the Galois group obtained above from a $p$-adic Galois representation is indeed a $p$-adic Lie group.
	
	\begin{theorem}\label{t:subgroup}
		Every closed subgroup $\Gamma$ of a $p$-adic Lie group $G$ is a $p$-adic Lie group. If $H$ is an open uniform subgroup of $G$, then $H^k\cap \Gamma$ is open and uniform in $\Gamma$ for $k$ sufficiently large, with associated Lie filtration
		\begin{align*}
			\Gamma	\geq	H^{p^k}\cap \Gamma	\geq	H^{p^{k+1}}\cap \Gamma	\geq	\cdots.
		\end{align*}
	\end{theorem}
	
	\begin{proof}
		For brevity, let $\Gamma(i)=H^{p^i}\cap \Gamma$. Then $\Gamma(i)$ is a closed subgroup of the uniform group $H$. This suffices to deduce that $\Gamma(i)$ is finitely generated, as in \cite{Dixon} Theorem 3.8. We aim to show that the $p^\text{th}$ power map induces an isomorphism
		\begin{align*}
			\Gamma(i)/\Gamma(i+1)	\to \Gamma(i+1)/\Gamma(i+2)
		\end{align*}
		for $i\gg 0$. Consider the composition
		\begin{align*}
			\Gamma(i)/\Gamma(i+1)	\cong	H^{p^{i+1}}(H^{p^i}\cap \Gamma)/H^{p^{i+1}}	\hookrightarrow	H^{p^i}/H^{p^{i+1}}	\xrightarrow{p}	H^{p^{i+1}}/H^{p^{i+2}},
		\end{align*}
		where the first map is the second isomorphism theorem, and the final map is the $p^\text{th}$-power isomorphism. It is straightforward to see that the image lies in $H^{p^{i+2}}(H^{p^{i+1}}\cap \Gamma)/H^{p^{i+2}}\cong \Gamma(i+1)/\Gamma(i+2)$, and so we obtain an injection
		\begin{align*}
			\Gamma(i)/\Gamma(i+1)	\hookrightarrow \Gamma(i+1)/\Gamma(i+2).
		\end{align*}
		Since each of these subquotients is isomorphic to a subgroup of $H^{p^i}/H^{p^{i+1}} \cong (\Z/p\Z)^{\dim(G)}$, they must stabilize for $i\gg 0$, after which the injections become isomorphisms.
	\end{proof}
	
	When $G$ is a closed subgroup of $\mathrm{GL}_n(\mathbb{Z}_p)$ (as in the case of $\mathrm{Gal}(L/K)$ above), Theorem \ref{t:subgroup} and Example \ref{e:gln} imply that the subgroups
	\begin{equation*}
		\ker(G\to	\mathrm{GL}_n(\mathbb{Z}/p^i\mathbb{Z}))
	\end{equation*}
	eventually form a Lie filtration of $G$, i.e. we can think of the Lie filtration as a family of congruence subgroups of $G$ modulo high powers of $p$.
	
	\section{Reduction to the Local Case}\label{s:local}
		
		We will now begin proving the main result by reducing the the case of a $p$-adic Lie tower of local fields of characteristic $0$. In particular we will show that Theorem~\ref{main} follows from a similar stability result for the different $\frakD_{K(i)/K}$, which can be computed using local ramification groups. Consider again the $p$-adic Lie tower of number fields
		\begin{align*}
			L	\supseteq	\cdots	\supseteq	K(1)	\supseteq	K(0)	\supseteq K.
		\end{align*}
		For a fixed prime $\frakp$ in $K$, choose a compatible system of primes $\frakp_i$ in $K(i)$ lying above $\frakp$. Recall that the different and discriminant are ideals in $\calO_{K(i)}$ and $\calO_K$, respectively, and are related by
		\begin{align*}
			N_{K(i)/K}\frakD_{K(i)/K}	&=	\frakd_{K(i)/K}.
		\end{align*}
		In particular, the $\frakp$-adic valuation of the discriminant can be written as
		\begin{align*}
			v_{\frakp}(\frakd_{K(i)/K})	&=	f_iv_{\frakp_i}(\frakD_{K(i)/K}),
		\end{align*}
		where $f_i=f(\frakp_i|\frakp)$ is the residue field degree. We will see that we can rewrite the $f_i$ in terms of the ramification indices $e_i=e(\frakp_i|\frakp)$. Consider the tower obtained after passing to completions
		\begin{align*}
			\cdots	\supseteq	K(1)_{\frakp_1}	\supseteq	K(0)_{\frakp_0}	\supseteq K_{\frakp}
		\end{align*}
		Write $L_\frakp$ for the compositum of this tower. Let $G_\frakp=\Gal(L_\frakp/K_\frakp)$ and similarly $G_{\frakp_i}=\Gal(L_\frakp/K_{\frakp_i})$ (note these are all open subgroups of $G_\frakp$). The following Lemma guarantees that $G_\frakp$ also has the structure of a $p$-adic Lie group:
		
		\begin{lemma}
			There is a closed embedding $\iota:G_\frakp \hookrightarrow G$ such that
			\begin{enumerate}
				\item	The image of $G_\frakp$ is the subgroup
					\begin{align*}
						\Gamma	&=	\{\sigma \in G:\sigma|_{K(i)}(\frakp_i)=\frakp_i\text{ for all $i$}\}.
					\end{align*}
				\item	The image of $G_{\frakp_i}$ is $\Gamma \cap G(i)$.	
			\end{enumerate}
		\end{lemma}
		\begin{proof}
			Fix a system of compatible embeddings $K(i)\hookrightarrow K(i)_{\frakp_i}$, giving an embedding $L\hookrightarrow L_\frakp$. The map $\iota$ is then given by restriction to $L$. This is clearly a continuous homomorphism, and is injective by density of $L$ in $L_\frakp$. Since $G_\frakp$ is compact, its image is closed in $G$.
			
			For (i), note that any $\sigma \in G_\frakp$ must fix each $\frakp_i$, so that $\iota(G_\frakp)\subseteq \Gamma$. Conversely if $\tau \in \Gamma$, then $\tau|_{K(i)}$ is an isometry of $K(i)$ under the $\frakp_i$-adic metric. By density of $K(i)$ in $K(i)_{\frakp_i}$, this extends uniquely to an element $\tau_i \in \Gal(K(i)_{\frakp_i}/K_\frakp)$. By uniqueness and the fact that all such elements come from $\tau$, they must be compatible. Therefore the $\tau_i$ extend to an element of $G_\frakp$ whose restriction to $L$ is $\tau$. To see (ii), observe that $G_{\frakp_i}$ is precisely the subgroup of $G_\frakp$ fixing $K(i)_{\frakp_i}$. Again by density, any $\sigma \in G_\frakp$ fixes $K(i)_{\frakp_i}$ if and only if $\sigma$ fixes $K(i)$, i.e. the image of $\sigma$ lies in $G(i)$.
		\end{proof}
		
		In light of the Lemma and Theorem~\ref{t:subgroup}, $G_\frakp$ is a $p$-adic Lie group with associated Lie filtration $G_{\frakp_{k+i}}$ for some $k\geq 0$. Hence the completions $K(k+i)_{\frakp_{k+i}}$ form a $p$-adic Lie tower of local fields over $K_\frakp$. Under the above embedding $G_\frakp/G_{\frakp_i}$ is isomorphic to the decomposition group $D(\frakp_i|\frakp)=\Gamma G(i)/G(i)$, which has order $e_if_i$. Letting $d_\frakp=\dim(G_\frakp)$, we can rewrite $f_{k+i}$ as
		\begin{align*}
			\frac{|G_\frakp/G_{\frakp_{k+i}}|}{e_{k+i}}	=	\frac{[K(k+i)_{\frakp_{k+i}}:K_\frakp]}{e_{k+i}}	=	\frac{[K(k)_{\frakp_k}:K_\frakp]p^{d_\frakp i}}{e_{k+i}}	= \frac{e_kf_kp^{d_\frakp i}}{e_{k+i}}.
		\end{align*}
		Thus for $i>k$,
		\begin{align}\label{discDiff}
			v_{\frakp}(\frakd_{K(i)/K})	&=	f_iv_{\frakp_i}(\frakD_{K(i)/K})	 =	\frac{e_kf_kp^{d_\frakp(i-k)}}{e_{i}}v_{\frakp_i}(\frakD_{K(i)_{\frakp_i}/K_\frakp}).
		\end{align}
		Therefore the proof of Theorem~\ref{main} is reduced to the computation of local ramification groups. From this point forward, we no longer need to consider the global case, so we will drop the $\frakp$-subscripts, replacing $K$ by $K_\frakp$ and similarly for each $K(i),L$. We will write $v_K$ instead of $v_\frakp$ for the unique normalized valuation of $K$, and similarly for all other local fields. As discussed in the following section, the inertia subgroup $G^0$ of $G=\Gal(L/K)$ is a closed Lie subgroup, and we let $d=\dim(G^0)$. We will see as an immediate consequence of Theorem~\ref{ram} that there is a stable formula for the $e_i$ in terms of $p^i$. The full result will then follow from the following more explicit statement:
		
		\begin{theorem}\label{t:main2}With notation as above, for $i\gg 0$
				\begin{align*}
					v_{K(i)}(\frakD_{K(i)/K})	&=	Cip^{di}+Ap^{di}+B,
				\end{align*}
				where $A,B,C$ are constants. If $L/K$ is totally ramified, then $C=e_0$.
		\end{theorem}
	
	\section{Ramification Theory}\label{s:ram}
	
		In this section we discuss the ramification theory of local fields relevant to the proof of Theorem~\ref{t:main2} in Section~\ref{s:proof}. A more comprehensive account can be found in Chapter IV of~\cite{Serre}. Let $K$ denote any local field, and suppose temporarily that $L$ is a finite Galois extension of $K$. It is well known that $\calO_L$ is a free rank-1 $\calO_K$-algebra, so that $\calO_L=\calO_K[\alpha]$ for some $\alpha$. The valuation on $L$ determines a family of subgroups of $G=\Gal(L/K)$ called the \emph{(lower) ramification groups} of $L/K$:
		\begin{align*}
			G_r	&=	\{\sigma \in G:v_L(\sigma\alpha-\alpha)\geq r+1\},
		\end{align*}
		indexed by real $r\geq -1$. This is easily seen to be a descending filtration of $G$ by normal subgroups, independent of the choice of $\alpha$. The group $G_0$ is called the \emph{inertia group} of $G$, and we say the extension $L/K$ is \emph{totally ramified} if $G_0=G$ and \emph{unramified} if $G_0=1$. The \emph{ramification index} of $L/K$ is defined to be $e_{L/K}=|G_0|$.
		
		For any $H\leq G$, $L$ is Galois over $F=L^H$ with Galois group $H$. The ramification groups of $L/F$ and $L/K$ are related via
		\begin{align*}
			H_r	&=	G_r	\cap H.
		\end{align*}
		If additionally $H$ is normal in $G$, then $F$ is Galois over $K$ with Galois group $G/H$. There is \emph{a priori} no relationship between the ramification groups of $F/K$ and $L/K$, but we can obtain one after re-indexing: Define the function $\varphi:[-1,\infty)\to[-1,\infty)$ by
		\begin{align*}
			\varphi_{L/K}(r)	&=	\begin{cases}
					\int_{0}^r	\frac{dw}{[G_0:G_w]}	&	\text{if }v	\geq 0	\\
					v	&	\text{if }-1\leq v <	0
				\end{cases}.
		\end{align*}
		This is continuous and strictly increasing, with a continuous inverse which we denote by $\psi_{L/K}$. For any $v\geq -1$, we define the $v^\text{th}$ \emph{(upper) ramification group} of $L/K$ to be
		\begin{align*}
			G^v	&=	G_{\psi_{L/K}(v)}.
		\end{align*}
		Note that $G^v=G_v$ for $v\leq 0$, but these groups differ in general for larger $v$. The ramification groups in upper numbering satisfy the desired property with respect to quotients, namely if $H$ is normal in $G$, then
		\begin{align*}
			(G/H)^v	&=	G^vH/H.
		\end{align*}
		We can now extend the ramification groups to possibly infinite Galois extensions: given any $L$ Galois over $K$, let
		\begin{align*}
			\Gal(L/K)^v	&=	\varprojlim_{F/K\text{ finite}}	\Gal(F/K)^v.
		\end{align*}
		Compatibility with quotients guarantees that the $\Gal(F/K)^v$ form a projective system, so that this limit is well defined and agrees with the former in the finite case. It is immediate that the ramification groups are always closed, normal subgroups in $\Gal(L/K)$. If additionally the ramification groups are open in $\Gal(L/K)$ (i.e. they have finite index), we can extend the functions $\psi$ and $\varphi$ to the infinite case:
		\begin{align*}
			\psi_{L/K}(v)	&=	\begin{cases}
					\int_0^\infty[G:G^w]dw	&	\text{if }v	\geq 0	\\
					v	&	\text{if }-1\leq v <	0
				\end{cases}.
		\end{align*}
		This is again bijective, so we may define $\varphi_{L/K}=\psi_{L/K}^{-1}$.
		
		In light of the discussion in Section~\ref{s:local}, we are primarily interested in the behavior of the \emph{different} in certain towers of local fields. This is an arithmetic invariant of a finite extension $F/K$ obtained from the ramification groups, defined to be the ideal in $\calO_F$ with valuation\footnote{For the equivalence of this definition with the usual definitions of the different, see \cite{Fontaine}, Section 0.3.5.}
		\begin{align}
			v_F(\frakD_{F/K})	&=	e_{F/K}\int_{-1}^\infty(1-|\Gal(F/K)^v|^{-1})dv.\label{diff}
		\end{align}
		Note that this is well defined as $F/K$ is finite, and always integral. If we return to the case of a $p$-adic Lie tower,
		\begin{align*}
			L	\supseteq	\cdots	\supseteq	K(1)	\supseteq	K(0)	\supseteq K,
		\end{align*}
		then the different of $K(i)/K$ is given by
		\begin{align}
			v_{K(i)}(\frakD_{K(i)/K})	&=	e_{K(i)/K}\int_{-1}^\infty(1-|\Gal(K(i)/K)^v|^{-1}dv \nonumber\\
				&=	e_{K(i)/K}\int_{-1}^\infty(1-|(G/G(i))^v|^{-1}dv	\nonumber\\
				&=	e_{K(i)/K}\int_{-1}^\infty(1-[G^vG(i):G(i)]^{-1}dv.	\label{diff2}
		\end{align}
		Thus, our goal is to understand the relation between the Lie filtration $G(i)$ and the ramification groups $G^v$. Before proving Theorem~\ref{main}, we cite some necessary information regarding ramification in local fields:
		
		\begin{theorem}\label{t:diff}
			Let $F/K$ be a finite Galois extension contained in $L$.
			\begin{enumerate}
				\item	$\varphi_{L/K}	=	\varphi_{F/K}\circ \varphi_{L/F}\text{ and }\psi_{L/K}	=	\psi_{L/F}\circ \psi_{F/K}$.
				\item	If $L/K$ is finite, then $\frakD_{L/K}=\frakD_{L/F}\frakD_{F/K}$.
			\end{enumerate}
		\end{theorem}
		\begin{proof}
			The first statement is given in \cite{Fontaine} Proposition 0.65, with the infinite case discussed in the following section. Statement (ii) can be found e.g. in \cite{Serre} III. Section 4.
			\end{proof}
		
	\section{Proof of the Main Result}\label{s:proof}
	
		In this section we will proceed with the proof of Theorem~\ref{t:main2}. Suppose again that $K$ is a local field of characteristic $0$, and that $L/K$ is Galois with $G=\Gal(L/K)$ a $p$-adic Lie group. Fix a Lie filtration $G(i)$ of $G=\Gal(L/K)$. Set $e=e_{K/\Q_p}$ and $d=\dim(G^0)$. As a first step, we will reduce to the case where $L/K$ is totally ramified. Letting $K^\text{unr}$ denote the maximal unramified extension of $K$, define
		\begin{align*}
			K'	&=	\widehat{K^\text{unr}},~~K(i)'=\widehat{K(i)K^\text{unr}},~~L'=\bigcup_i	K(i)',
		\end{align*}
		(where the hat denotes completion). We now have a totally ramified tower of local fields
		\begin{align*}
			L'	\supseteq	\cdots	K(1)'	\supseteq	K(0)'	\supseteq	K'.
		\end{align*}
		
		\begin{proposition}\label{ram}
			Let $G'=\Gal(L'/K')$. Then
			\begin{enumerate}
				\item	There is an isomorphism $\varphi:G^0\to G'$ preserving ramification groups, that is
					\begin{align*}
						(G')^v	=	\varphi(G^v).
					\end{align*}
				\item	$G'$ is a $p$-adic Lie group, and there is a Lie filtration
					\begin{align*}
						G(i)'=\Gal(L'/K(k+i)')
					\end{align*}
					for a fixed $k$. Thus the $K(k+i)'$ form a $p$-adic Lie tower.
				\item	For all $i$,
					\begin{align*}
						v_{K(i)}(\frakD_{K(i)/K})	=	v_{K(i)'}(\frakD_{K(i)'/K'}).
					\end{align*}
					In particular, if the tower $K(k+i)'$ is stable, then so is $K(i)$.
			\end{enumerate}
		\end{proposition}
		\begin{proof}
			\begin{enumerate}
				\item	Recall that $G^v$ is the inverse limit of $\Gal(F/K)^v$, where $F$ runs over finite intermediate Galois extensions of $L/K$. For any such $F$, there exists $\alpha$ such that $\calO_F=\calO_{F\cap K^\text{unr}}[\alpha]$. Then $FK^\text{unr}$ is generated over $K^\text{unr}$ by $\alpha$ as well. Thus there is an isomorphism
				\begin{align*}
					\Gal(F/K)^0	&\to	\Gal(FK^\text{unr}/K^\text{unr})	\\
					\sigma	&\mapsto	\hat{\sigma}
				\end{align*}
				where $\hat{\sigma}$ is determined by $\hat{\sigma}(\alpha)=\sigma(\alpha)$. Since the ramification groups are determined by the action of the Galois group on $\alpha$, this isomorphism preserves ramification groups. Passing to completion will not change the Galois group, and so this extends to an isomorphism $\Gal(F/K)^0\cong \Gal(\widehat{FK^\text{unr}/K^\text{unr}})$. It is straightforward to see that all such isomorphisms are compatible, and therefore extend to a continuous ramification-preserving isomorphism $\varphi:G^0\to G'$.
				
				\item	Since $G^0$ is a closed subgroup of $G$, $G^0$ is a $p$-adic Lie group by Theorem~\ref{t:subgroup}. It follows that there is some $k$ such that $G^0\cap G(k+i)$ is a Lie filtration of $G^0$. Under the isomorphism $\varphi$, we see that $G'$ is a $p$-adic Lie group and there is a Lie filtration
					\begin{align*}
						G'(i)	&=	\varphi(G^0	\cap G(k+i))	\\
							&=	\varphi(\Gal(L/K(k+i)K^\text{unr}))	\\
							&=	\Gal(L'/K(k+i)').
					\end{align*}
					It is worth noting at this point that for $i>k$, $e_i=[K(i)':K']=p^{d(i-k)}[K(k)':K']$ is also given by a polynomial in $p^i$. This implies that stability of the different and discriminant are equivalent, by (\ref{discDiff}).
				\item	Using Theorem~\ref{t:diff} (i), we see for $i>k$
					\begin{align*}
						v_{K(i)}(\frakD_{K(i)/K})	&=	v_{K(i)}(\frakD_{K(i)/K(i)\cap K^\text{unr}}\frakD_{K(i)\cap K^\text{unr}/K})	\\
							&=	v_{K(i)}(\frakD_{K(i)/K(i)\cap K^\text{unr}})+v_{K(i)}(\frakD_{K(i)\cap K^\text{unr}/K})	\\
							&=	v_{K(i)}(\frakD_{K(i)/K(i)\cap K^\text{unr}})	\\
							&=	v_{K(i)'}(\frakD_{K(i)'/K'}),
					\end{align*}
					Where the last equality holds since the two extensions have the same ramification groups. Finally, assume that $K(k+i)'$ is a stable tower. Then for $i\gg 0$ (and $i>k$),
					\begin{align*}
						v_{K(i)}(\frakD_{K(i)/K})	&=	v_{K(k+(i-k))'}(\frakD_{K(k+(i-k))'/K'})	\\
							&=	C(i-k)p^{d(i-k)}+Ap^{d(i-k)}+B	\\
							&=	\frac{C}{p^{dk}}ip^{di}+\frac{A-Ck}{p^{dk}}p^{di}+B.
					\end{align*}
			\end{enumerate}
		\end{proof}
		
		In light of Proposition~\ref{ram}, we replace $K$, $K(i)$, and $L$ with $K'$, $K(k+i)'$, and $L'$ respectively, reducing to the case of a totally ramified tower. In this situation, there is agreement between the Lie filtration $G(i)$ and the ramification groups of $G$:
		
		\begin{theorem} (Sen). There exists a constant $c$ such that for all $i\geq 0$,
				\begin{align*}
						G^{ie+c}\leq	G(i)	\leq	G^{ie-c}.
				\end{align*}
				Furthermore, $(G^v)^p=G^{v+e}$ for $v\gg 0$.
		\end{theorem}
		
		It is immediate from Sen's theorem that
		\begin{align}\label{est}
			G\lceil\tfrac{v+c}{e}\rceil	\leq	G^v	\leq	G\lfloor\tfrac{v-c}{e}\rfloor,
		\end{align}
		which implies that the ramification groups have finite index in $G$, and hence are open. Our next step is to produce a second Lie filtration $G[i]$ of $G$ from the ramification groups.
		
		\begin{lemma}
			For some $v_0> 0$, $G^{v_0}$ is a uniform pro-$p$ group.
		\end{lemma}
		\begin{proof}
			We may assume $v$ sufficiently large so that $(G^v)^p=G^{v+e}$. As $G^v$ is a closed subgroup of the uniform group $G(0)$, $G^v$ is finitely generated. By Sen's theorem, we have a chain of inclusion
			\begin{align}\label{chain}
				G\lfloor \tfrac{v-c}{e}\rfloor	\geq	G^v	\geq	G^{v+e}	\geq	G\left(1+\lceil \tfrac{v+c}{e}\rceil\right).
			\end{align}
			Observe that $1+\lceil \tfrac{v+c}{e}\rceil-\lfloor \tfrac{v-c}{e}\rfloor\leq 2+2c/e$. We take $v_0$ sufficiently large so that $2+2c/e \leq \lfloor\tfrac{v_0-c}{e}\rfloor+1$ and the quotient
			\begin{align*}
				G\lfloor \tfrac{v_0-c}{e}\rfloor/G\left(1+\lceil \tfrac{v_0+c}{e}\rceil\right)
			\end{align*}
			is abelian (as in Proposition~\ref{props}). Thus $G^{v_0}/(G^{v_0})^p$ is abelian as well. The shift isomorphism then induces an isomorphism
			\begin{align*}
				G^{v_0+ke}/G^{v_0+(k+1)e}\to G^{v_0+(k+1)e}/G^{v_0+(k+2)e}
			\end{align*}
			for all $k\geq 0$.
		\end{proof}
		
		As in Section~\ref{s:Lie} above, a choice of open, uniform subgroup produces a Lie filtration of $G$. Letting $G[i]=G^{v_0+ie}$ of $G$, we obtain a second $p$-adic Lie tower $K[i]=L^{G[i]}$ over $K$. This decomposition has the advantage of being defined in terms of the ramification groups, allowing for a more direct computation of the different.
		
		\begin{proposition}
			The tower $K[i]/K$ is stable.
		\end{proposition}
		\begin{proof}
			Note that $G^vG[i]=G^v$ when $v\leq v_0+ie$ and is trivial otherwise. We proceed with (\ref{diff2}),
			\begin{align*}
				v_{K}(\frakD_{K[i]/K})	&=	\int_{-1}^\infty\left(	1-[G^vG[i]:G[i]]^{-1}	\right)dv	\\
					&=	\int_{-1}^{v_0+ie}\left(	1-[G^v:G[i]]^{-1}	\right)dv	\\
					&=	ie+v_0+1-\int_{-1}^{v_0}[G^v:G[i]]^{-1}dv-\int_{v_0}^{v_0+ie}[G^v:G[i]]^{-1}dv.
			\end{align*}
			Now
			\begin{align*}
				\int_{-1}^{v_0}[G^v:G[i]]^{-1}dv	&=	[G[0]:G[i]]^{-1}\int_{-1}^{v_0}[G^v:G[0]]^{-1}dv	\\
					&=	B'p^{-di}
			\end{align*}
			where $B'$ is constant. For the remaining term, we split the interval $[v_0,v_0+ie]$ into $i$ intervals, each of length $e$
			\begin{align*}
				\int_{v_0}^{v_0+ie}	[G^v:G[i]]^{-1}dv	&=	\sum_{k=0}^{i-1}\int_0^e [G^{v_0+ke+v}:G[i]]^{-1}dv	\\
					&=	\sum_{k=0}^{i-1}\int_0^e [G^{v_0+ke+v}:G[k+1]]^{-1}[G[k+1]:G[i]]^{-1}dv	\\
					&=	\sum_{k=0}^{i-1}p^{-d(i-k-1)}\int_0^e\frac{1}{[G^{v_0+ke+v}:G[k+1]]}dv	\\
					&=	\left(\int_0^e\frac{1}{[G^{v_0+v}:G[1]]}dv\right)\sum_{k=0}^{i-1}p^{-d(i-k-1)}	\tag{\textasteriskcentered}\label{int}\\
					&=	A'\frac{p^{-di}-1}{p^{-d}-1},
			\end{align*}
			where in (\ref{int}), we use the fact that the $k^\text{th}$ shift isomorphism
			\begin{align*}
				G[0]/G[1]	\to	G[k]/G[k+1]
			\end{align*}
			induces an isomorphism
			\begin{align*}
				G^{v_0+v}/G[1]	\to	G^{v_0+ke+v}/G[k+1].
			\end{align*}
			Finally, we obtain
			\begin{align*}
				v_{K[i]}(\frakD_{K[i]/K})	&=	e_{K[i]/K}v_{K}(\frakD_{K[i]/K})	\\
					&=	e_{K[0]}ip^{di}+Ap^{di}+B
			\end{align*}
			for some constants $A,B$.
		\end{proof}
		
		The full result will follow by comparing the towers $K(i)$ and $K[i]$ in high degrees. Let $i_0$ be sufficiently large so that $G[0]\geq G(i_0)$. Assume that $c>e$, so that there is a chain of inclusions
		\begin{align*}
			G\left(i_0-\lceil\tfrac{2c}{e}\rceil\right)	\geq	G^{i_0e-c}\geq	G[k_{i_0}]	\geq	G(i_0)	\geq	G^{i_0e+c}	\geq	G\left(	i_0+\lceil\tfrac{2c}{e}\rceil	\right).
		\end{align*}
		for some $k_{i_0}$. If additionally we take $i_0 > 2\lceil\tfrac{2c}{e}\rceil$ large enough, then the quotient
		\begin{align*}
			G\left(	i_0+\lceil\tfrac{2c}{e}\rceil	\right)/G\left(i_0-\lceil\tfrac{2c}{e}\rceil\right)
		\end{align*}
		is abelian, and there is a shift isomorphism induced by $p^\text{th}$ powers. Letting $k_i=k_{i_0}+(i-i_0)$, the shift isomorphism induces an isomorphism
			\begin{align}\label{shift}
				G[k_{i_0}]/G(i_0)	\to	G[k_i]/G(i).
			\end{align}
		The different is then given by
			\begin{align*}
				v_{K(i)}(\frakD_{K(i)/K})	&=	v_{K(i)}\left(	\frakD_{K(i)/K[k_i]}\frakD_{K[k_i]/K}	\right)	\\
					&=	v_{K(i)}\left(	\frakD_{K(i)/K[k_i]})+v_{K(i)}(\frakD_{K[k_i]/K}	\right).
			\end{align*}
			 The second term is straightforward to compute: the valuation $v_{K(i)}$ can be rewritten as $[K(i):K[k_{i}]]v_{K[k_i]}=[K(i_0):K[k_{i_0}]]v_{K[k_i]}$. Thus
			\begin{align*}
				v_{K(i)}(\frakD_{K[k_i]/K}	)
					&=	[K(i_0):K[k_{i_0}]](e_{K[0]}k_ip^{dk_i}+A'p^{dk_i}+B')\\
					&=	p^{d(i_0-k_{i_0})}\frac{[K(0):K]}{[K[0]:K]}(e_{K[0]}k_ip^{d(k_{i_0}-i_0)}p^{di}+A''p^{di}+B').
			\end{align*}
			As expected, the leading term is
			\begin{align*}
				\frac{[K(0):K]}{[K[0]:K]}e_{K[0]}ip^{di}	&=	e_{K(0)}ip^{di}.
			\end{align*}
			Note that the shift isomorphism (\ref{shift}) does not preserve ramification groups in general. In order to compute $\frakD_{K(i)/K[k_i]}$, we need to relate the ramification groups of $G[k_{i_0}]/G(i_0)$ and $G[k_i]/G(i)$.
			
			\begin{lemma}\label{lemmaa}
				For $v\gg 0$, the $p^{i-i_0}$-power map induces an isomorphism
				\begin{align*}
					(G[k_{i_0}]/G(i_0))^v	\to	(G[k_i]/G(i))^x,
				\end{align*}
				where $x=\psi_{K[k_i]/K}(\varphi_{K[k_{i_o}]/K}(v)+(i-i_0)e)$. More precisely, this occurs when $x\geq \psi_{K[k_i]/K}(v_0+(i-i_0)e)$.
			\end{lemma}
			\begin{proof}
				For any $i$, the $v^\text{th}$ ramification group of $G[k_i]/G(i)$ is
				\begin{align*}
					(G[k_i]/G(i))^v	&=	G[k_i]^vG(i)/G(i)	\\
						&=	G[k_i]_{\psi_{L/K[k_i]}(v)}G(i)/G(i)	\\
						&=	(G_{\psi_{L/K[k_i]}(v)} \cap G[k_i])G(i)/G(i)	\\
						&=	(G^{\varphi_{L/K}(\psi_{L/K[k_i]}(v))}\cap G[k_i])G(i)/G(i)	\\
						&=	(G^{\varphi_{L/K[k_i]}(v)}\cap G[k_i])G(i)/G(i).
				\end{align*}
				The $p^{i-i_0}$-power map induces an isomorphism $G[k_{i_0}]/G(i_0)\to G[k_i]/G(i)$. When $\varphi_{L/K[k_{i_0}]}(v) \geq v_0$, this maps $(G[k_{i_0}]/G(i_0))^v$ to
				\begin{align*}
					&(G^{\varphi_{L/K[k_{i_0}]}(v)+(i-i_0)e}\cap G[k_i])G(i)/G(i)	\\
					&=	(G[k_i]/G(i))^{\psi_{K[k_i]/K}(\varphi_{L/K[k_{i_0}]}(v)+(i-i_0)e)}	\\
					&=	(G[k_i]/G(i))^x
				\end{align*}
				as desired. Note that the condition $\varphi_{L/K[k_{i_0}]}(v) \geq v_0$ is equivalent to the statement that $x\geq \psi_{K[k_i]/K}(v_0+(i-i_0)e)$.
			\end{proof}
			
			We will need a description of the functions $\psi_{K[k_i]/K}$:
			
			\begin{lemma}\label{l:psi}
				Let $x\geq v_0$ and write $x=v_0+je+x_0$, where $x_0 \in [0,e)$. Then for any $i\geq i_0$
				\begin{align*}
					\psi_{K[k_i]/K}(x)	&=	\begin{cases}
							R_1p^{dj}g(x_0)+R_2p^{dj}+R_3	&	x<v_0+k_ie	\\
							R_1'p^{di}((j-i)e+x_0)+R_2'p^{di}+R_3'	&	x\geq v_0+k_ie
						\end{cases},
				\end{align*}
				where $g$ satisfies $g(0)=0$. When it exists, the derivative is given by
				\begin{align*}
					\psi_{K[k_i]/K}(x)	&=	\begin{cases}
							R_1[G^{v_0}:G^{v_0+x_0}]p^{dj}	&	x\leq v_0+k_ie	\\
							R_1'p^{di}	&	x> v_0+k_ie
						\end{cases}.
				\end{align*}
			\end{lemma}
			\begin{proof}
				We can compute directly
				\begin{align}
					\psi_{K[k_i]/K}(x)	&=	\int_0^x [G/G[k_i]:(G/G[k_i])^w]	dw	\nonumber\\
						&=	\int_0^{v_0+je+x_0} [G/G[k_i]:G^wG[k_i]/G[k_i]]dw\label{phiint}
				\end{align}
				In the first case (when $x\leq v_0+k_ie$), $G^w \geq G^{v_0+k_ie}=G[k_i]$, and the integrand simplifies to $[G:G^w]$. We split the integral into the intervals $[0,v_0]$, $[v_0,v_0+je]$, and $[v_0+je,v_0+je+x_0]$. Now
				\begin{align*}
					\int_0^{v_0}[G:G^w]dw = \psi_{L/K}(v_0)
				\end{align*}
				is constant, and
				\begin{align*}
					\int_{v_0+je}^{v_0+je+x_0}[G:G^w]dw	&=	[G:G^{v_0}]p^{dj}\int_{v_0+je}^{v_0+je+x_0}[G[j]:G^w]dw	\\
						&=	[G:G^{v_0}]p^{dj}\int_{v_0}^{v_0+x_0}[G[0]:G^w]dw	\\
						&=	[G:G^{v_0}]p^{dj}g(x_0)
				\end{align*}
				(note that we take advantage of the shift isomorphism). The middle interval can be further split into $j-1$ subintervals:
				\begin{align*}
					\int_{v_0+}^{v_0+je}[G:G^w]dw	&=	\sum_{k=0}^{j-1}\int_{v_0+ke}^{v_0+(k+1)e} [G:G^w]dw	\\
						&=	\sum_{k=0}^{j-1}[G:G^{v_0}]p^{dk}\int_{v_0+ke}^{v_0+(k+1)e} [G[k]:G^w]dw	\\
						&=	\left(	\int_{v_0}^{v_0+e} [G[0]:G^w]dw	\right)[G:[0]]\sum_{k=0}^{j-1}p^{dk}	\\
						&=	g(e)[G[0]:G^w]\frac{p^{dj}-1}{p-1}.
				\end{align*}
				This completes the first case. The second case is simpler: when $x> v_0+k_ie$, the integral (\ref{phiint}) becomes
				\begin{align*}
					&\psi_{K[k_i]/K}(v_0+k_ie)+\int_{v_0+k_ie}^{v_0+je+x_0} [G/G[k_i]:G^wG[k_i]/G[k_i]]dw	\\
					&=	R_2p^{dk_i}+R_3+\int_{v_0+k_ie}^{v_0+je+x_0}[G:[k_i]]dw	\\
					&=	R_2p^{dk_i}+R_3+[G:G[0]]p^{dk_i}((j-k_i)e+x_0).
				\end{align*}
				The functions $\psi_{K[k_i]/K}$ are continuous and piecewise differentiable, and the computation of their derivatives is straightforward from the above formulas.
			\end{proof}
			
			\begin{proposition}
				The tower $K(i)$ is stable.
			\end{proposition}
			\begin{proof}
				For brevity, write $H(i)=\Gal(K(i)/K[k_i])=G[k_i]/G(i)$ and $\psi_i=\psi_{K[k_i]/K}$. In light of the preceding discussion, all that remains is the computation of
				\begin{align}\label{difffinal}
					v_{K(i)}(	\frakD_{K(i)/K[k_i]})	&=	|H(i)|\int_{-1}^\infty(1-|H(i)^x|^{-1})dx.
				\end{align}
				Note that $|H(i)|=|H(i_0)|$ is independent of $i$. We will split the integral at the point $x'=\psi_{K[k_i]/K}(v_0+(i-i_0)e)$. Note that $x'<v_0+k_ie$, so if $x<x'$, $G^x \cap G[k_i] = G[k_i]$. The first integral is then
				\begin{align*}
					\int_{-1}^{\psi_i(v_0+(i-i_0)e)}(1-|H(i)^x|^{-1})dx	&=	\int_{-1}^{v_0+(i-i_0)e}(1-|H(i)^{\psi_i(x)}|^{-1})\psi_i'(x)dx
				\end{align*}
				using a standard change of variables. Following the proof of Lemma~\ref{lemmaa},
				\begin{align*}
					H(i)^{\psi_i(x)}=(G^x\cap G[k_i])G(i)/G(i) = G[k_i]/G(i) = H(i).
				\end{align*}
				Therefore, we obtain
				\begin{align*}
					\int_{-1}^{v_0+(i-i_0)e}(1-|H(i)|^{-1})\psi_i'(x)dx	&=	(1-|H(i_0)|^{-1})\int_{-1}^{v_0+(i-i_0)e}\psi_i'(x)dx.
				\end{align*}
				This remaining integral can be evaluated by splitting up the interval $[v_0,v_0+(i-i_0)e]$ into $i-i_0$ pieces, as in the proof of Lemma~\ref{l:psi}. We omit the details, but note that the result is of the form $Ap^{di}+B$ for constants $A,B$. Thus all that remains is the second half of the integral in (\ref{difffinal}). In this case, $x>x'$ and Lemma~\ref{lemmaa} applies. A simple verification shows that $H(i)^x=1$ if $x>\psi_i(ie+c)$, so the integral is
				\begin{align*}
					\int_{\psi_i(v_0+(i-i_0)e)}^{\psi_i(ie+c)} (1-|H(i)^x|^{-1})dx	&=	\int_{\psi_i(v_0+(i-i_0)e)}^{\psi_i(ie+c)} (1-|H(i_0)^{v(x)}|^{-1})dx,
				\end{align*}
				where $v(x)$ and $x$ are related as in Lemma~\ref{l:psi}. Again, we will apply a change of variables:
				\begin{align*}
					&\int_{v_0+(i-i_0)e}^{ie+c} (1-|H(i_0)^{v(\psi_i(u))}|^{-1})\psi_i'(u)du	\\
						&=	\int_{v_0+(i-i_0)e}^{ie+c} (1-|H(i_0)^{u-(i-i_0)e}|^{-1})\psi_i'(u)du	\\
						&=	\int_{v_0}^{c+ei_0} (1-|H(i_0)^u|^{-1})\psi_i'(u+(i-i_0)e)du
				\end{align*}
				Note that $u+(i-i_0)e\leq v_0+k_ie$ if and only if $u\leq v_0+k_{i_0}$, so we will split the integral at this point. The first piece is
				\begin{align*}
					\int_{v_0}^{v_0+k_{i_0}e} (1-|H(i_0)^u|^{-1})R_1p^{d(i-i_0)}du	&=	Ap^{di}
				\end{align*}
				for some constant $A$. The second piece is
				\begin{align*}
					\int_{v_0+k_{i_0}e}^{c+ei_0} (1-|H(i_0)^u|^{-1})R_1'p^{di}du 	&=	Bp^{di}
				\end{align*}
				for some $B$. This completes the proof.
			\end{proof}
		
	\bibliographystyle{plainnat}
	\bibliography{mybib}
	\nocite{*}
	
\end{document}